\batchmode
\makeatletter
\def\input@path{{\string"/Users/russw/Documents/Research/mypapers/A broad class of shellable lattices/\string"}}
\makeatother
\documentclass[12pt,oneside,english]{amsart}
\usepackage[T1]{fontenc}
\usepackage[latin9]{inputenc}
\usepackage{geometry}
\usepackage{babel}
\usepackage{array}
\usepackage{url}
\usepackage{enumitem}
\usepackage{amsthm}
\usepackage{amssymb}
\usepackage{graphicx}
\usepackage[numbers]{natbib}
\PassOptionsToPackage{normalem}{ulem}
\usepackage{ulem}
\usepackage[unicode=true,pdfusetitle,
 bookmarks=true,bookmarksnumbered=false,bookmarksopen=false,
 breaklinks=false,pdfborder={0 0 1},backref=false,colorlinks=false]
 {hyperref}
\usepackage{breakurl}

\makeatletter

\providecommand{\tabularnewline}{\\}

\numberwithin{equation}{section}
\numberwithin{figure}{section}
\theoremstyle{plain}
\newtheorem{thm}{\protect\theoremname}[section]
  \theoremstyle{definition}
  \newtheorem{defn}[thm]{\protect\definitionname}
  \theoremstyle{plain}
  \newtheorem{cor}[thm]{\protect\corollaryname}
  \theoremstyle{remark}
  \newtheorem{rem}[thm]{\protect\remarkname}
  \theoremstyle{plain}
  \newtheorem{prop}[thm]{\protect\propositionname}
  \theoremstyle{definition}
  \newtheorem{example}[thm]{\protect\examplename}
  \theoremstyle{plain}
  \newtheorem{lem}[thm]{\protect\lemmaname}
  \theoremstyle{remark}
  \newtheorem{notation}[thm]{\protect\notationname}

\usepackage{ae}
\usepackage{aecompl}

\makeatletter
\newcommand{\setlabel}[1]{\def\@currentlabel{#1}}
\makeatother

\makeatother

  \providecommand{\corollaryname}{Corollary}
  \providecommand{\definitionname}{Definition}
  \providecommand{\examplename}{Example}
  \providecommand{\lemmaname}{Lemma}
  \providecommand{\notationname}{Notation}
  \providecommand{\propositionname}{Proposition}
  \providecommand{\remarkname}{Remark}
\providecommand{\theoremname}{Theorem}

\begin{document}
\global\long\def\cosetlat{\mathcal{C}}

\global\long\def\sglat{\mathcal{L}}

\global\long\def\ff{\mathbb{F}}

\global\long\def\nn{\mathbb{N}}

\global\long\def\zz{\mathbb{Z}}

\global\long\def\normalin{\mathrel{\triangleleft}}

\global\long\def\semidirect{\rtimes}

\global\long\def\comma{{,}}

\global\long\def\join{\mathbin{\ast}}

\global\long\def\ordcong{\mathcal{O}}

\global\long\def\ordconv{\mathcal{O^{\mathrm{conv}}}}

\global\long\def\bot{\hat{0}}

\global\long\def\top{\hat{1}}

\global\long\def\subsetdot{\mathrel{\subset\!\!\!\!{\cdot}\,}}

\global\long\def\dotsupset{\mathrel{\supset\!\!\!\!\!\cdot\,\,}}

\global\long\def\linext{\mathcal{LE}}

\global\long\def\height{\operatorname{height}}

\title{A broad class of shellable lattices}

\author{Jay Schweig and Russ Woodroofe}

\address{Department of Mathematics, Oklahoma State University, Stillwater,
OK, 74078}

\email{jay.schweig@okstate.edu}

\urladdr{\url{https://math.okstate.edu/people/jayjs/}}

\address{Department of Mathematics \& Statistics, Mississippi State University,
Starkville, MS 39762}

\email{rwoodroofe@math.msstate.edu}

\urladdr{\url{http://rwoodroofe.math.msstate.edu/}}
\begin{abstract}
We introduce a new class of lattices, the \emph{modernistic }lattices,
and their duals, the \emph{comodernistic} lattices. We show that every
modernistic or comodernistic lattice has shellable order complex.
We go on to exhibit a large number of examples of (co)modernistic
lattices. We show comodernism for two main families of lattices that
were not previously known to be shellable: the order congruence lattices
of finite posets, and a weighted generalization of the $k$-equal
partition lattices. 

We also exhibit many examples of (co)modernistic lattices that were
already known to be shellable. To begin with, the definition of modernistic
is a common weakening of the definitions of semimodular and supersolvable.
We thus obtain a unified proof that lattices in these classes are
shellable. 

Subgroup lattices of solvable groups form another family of comodernistic
lattices that were already proven to be shellable. We show  not only
that subgroup lattices of solvable groups are comodernistic, but that
solvability of a group is equivalent to the comodernistic property
on its subgroup lattice. Indeed, the definition of comodernistic exactly
requires on every interval a lattice-theoretic analogue of the composition
series in a solvable group. Thus, the relation between comodernistic
lattices and solvable groups resembles, in several respects, that
between supersolvable lattices and supersolvable groups.
\end{abstract}

\maketitle

\section{\label{sec:Introduction}Introduction}

Shellings are a main tool in topological combinatorics. An explicit
shelling of a simplicial complex $\Delta$ simultaneously shows the
sequentially Cohen-Macaulay property, computes homotopy type, and
gives a cohomology basis. Frequently, a shelling also gives significant
insight into the homeomorphy of $\Delta$. The downside is that shellings
are often difficult to find, and generally require a deep understanding
of the complex.

In this paper, we describe a large class of lattices whose order complexes
admit shellings. The shellings are often straightforward to explicitly
write down, and so give a large amount of information about the topology
of the order complex. Included in our class of lattices are many examples
which were not previously understood to be closely related.

The question that first motivated this research project involved shelling
a particular family of lattices. The \emph{order congruence} \emph{lattice}
$\ordcong(P)$ of a finite poset $P$ is the subposet of the partition
lattice consisting of all equivalence classes arising as the level
sets of an order-preserving function. Order congruence lattices interpolate
between Boolean lattices and partition lattices, as we will make precise
later. Such lattices were already considered by Sturm in \citep{Sturm:1972a}.
More recently, Körtesi, Radeleczki and Szilágyi showed the order congruence
lattice of any finite poset to be graded and relatively complemented
\citep{Kortesi/Radeleczki/Szilagyi:2005}, while Jen\v{c}a and Sarkoci
showed such lattices to be Cohen-Macaulay \citep{Jenca/Sarkoci:2014}.

The Cohen-Macaulay result naturally suggested to us the question of
whether every order congruence lattice is shellable. After proving
the answer to this question to be ``yes'', we noticed that our techniques
apply to a much broader class of lattices. Indeed, a large number
of the lattices previously shown to be shellable lie in our class.
Thus, our main result (Theorem~\ref{thm:MainThm} below) unifies
numerous results on shellability of order complexes of lattices, in
addition to proving shellability for new examples. It is our belief
that our results will be useful to other researchers. Finding a shelling
of a lattice can be a difficult problem. In many cases, showing that
a lattice is in our class may be simpler than constructing a shelling
directly.

All lattices, posets, simplicial complexes, and groups considered
in this paper will be finite. 

\subsection{Modernistic and comodernistic lattices}

We now define the broad class of lattices described in the title and
introduction. Our work relies heavily on the theory of modular elements
in a lattice. Recall that an element $m$ of a lattice $L$ is \emph{left-modular}
if whenever $x<y$ are elements of $L$, then the expression $x\vee m\wedge y$
can be written without parentheses, that is, that $(x\vee m)\wedge y=x\vee(m\wedge y)$.
An equivalent definition is that $m$ is left-modular if $m$ is not
the nontrivial element in the short chain of any pentagonal sublattice
of $L$; see Lemma~\ref{lem:NoPentagons} for a more precise statement. 

Our key object of study is the following class of lattices.
\begin{defn}
We say that a lattice $L$ is \emph{modernistic} if every interval
in $L$ has an atom that is left-modular (in the given interval).
We say that $L$ is \emph{comodernistic }if the dual of $L$ is modernistic,
that is, if every interval has a left-modular coatom.
\end{defn}
Our main theorem is as follows. (We will recall the definition of
a $CL$-labeling in Section~\ref{subsec:Prelims-CL} below.)
\begin{thm}
\label{thm:MainThm}If $L$ is a comodernistic lattice, then $L$
has a $CL$-labeling. 
\end{thm}
\begin{cor}
If $L$ is either comodernistic or modernistic, then the order complex
of $L$ is shellable.
\end{cor}
The $CL$-labeling is explicit from the left-modular coatoms, so Theorem~\ref{thm:MainThm}
also gives a method for computing the Möbius function of $L$. See
Lemma~\ref{lem:ComputingMobius} for details. 

We find it somewhat surprising that Theorem~\ref{thm:MainThm} was
not proved before now. We speculate that the reason may be the focus
of previous authors on atoms and $CL$-labelings, whereas Theorem~\ref{thm:MainThm}
requires dualizing exactly one of the two.
\begin{rem}
The name ``modernistic'' comes from contracting ``atomically modular''
to ``atomically mod''. Since atomic was a common superlative from
the the late 1940's, and since the mod (or modernistic) subculture
was also active at about the same time, we find the name to be somewhat
appropriate, as well as short and perhaps memorable.
\end{rem}

\subsection{Examples and applications}

Theorem~\ref{thm:MainThm} has a large number of applications, which
we briefly survey now. First, we can now solve the problem that motivated
the project.
\begin{thm}
\label{thm:OrdCongShellable}If $P$ is any poset, then the order
congruence lattice $\ordcong(P)$ is comodernistic, hence $CL$-shellable.
\end{thm}
We also recover as examples many lattices already known to be shellable.
The following theorem lists some of these, together with references
to papers where they are shown to be shellable. 
\begin{prop}
\label{prop:ComodList}The following lattices are comodernistic, hence
$CL$-shellable:

\begin{enumerate}
\item \label{enu:CML-Supersolvable}Supersolvable and left-modular lattices,
and their order duals \citep{Bjorner:1980,Liu:1999,McNamara/Thomas:2006}.
\item \label{enu:CML-DualGeom}Order duals of semimodular lattices \citep{Bjorner:1980}.
(I.e., semimodular lattices are modernistic.)
\item $k$-equal partition lattices \citep{Bjorner/Wachs:1996}, and their
type $B$ analogues \citep{Bjorner/Sagan:1996}.
\item \label{enu:CML-sglat} Subgroup lattices of solvable groups \citep{Shareshian:2001,Woodroofe:2008}.
\end{enumerate}
\end{prop}
We comment that many of these lattices are shown in the provided references
to have $EL$-labelings. Theorem~\ref{thm:MainThm} provides only
a $CL$-labeling. Since the $CL$-labeling constructed is explicit
from the left-modular elements, Theorem~\ref{thm:MainThm} provides
many of the benefits given by an $EL$-labeling. We do not know if
every comodernistic lattice has an $EL$-labeling, and leave this
interesting question open.

Experts in lattice theory will immediately recognize items (\ref{enu:CML-Supersolvable})
and (\ref{enu:CML-DualGeom}) from Proposition~\ref{prop:ComodList}
as being comodernistic. Theorem~\ref{thm:MainThm} thus unifies the
theory of these well-understood lattices with the more difficult lattices
on the list. The $CL$-labeling that we construct in the proof of
Theorem~\ref{thm:MainThm} can moreover be seen as a generalization
of the standard $EL$-labeling for a supersolvable lattice, further
connecting these classes of lattices.

We will prove that $k$-equal partition lattices and their type $B$
analogues are comodernistic in Section~\ref{sec:kEqualVariations}.
In Section~\ref{subsec:AffinityEqPartition} we will show the same
for a new generalization of $k$-equal partition lattices. The proofs
show, broadly speaking, that coatoms of (intervals in) these subposets
of the partition lattice inherit left-modularity from that in the
partition lattice.

Although modernism and comodernism give a simple and unified framework
for showing shellability of many lattices, not every shellable lattice
is (co)modernistic. For an easy example, the face lattice of an $n$-gon
has no left-modular elements when $n>3$, so is neither modernistic
nor comodernistic.

\subsection{Further remarks on subgroup lattices}

We can expand on the connection with group theory suggested by item
(\ref{enu:CML-sglat}) of Proposition~\ref{prop:ComodList}. For
a group $G$, the \emph{subgroup lattice} referred to in this item
consists of all the subgroups of $G$, ordered by inclusion; and is
denoted by $L(G)$. 
\begin{thm}
\label{thm:SgLatComodern}If $G$ is a group, then $G$ is solvable
if and only if $L(G)$ is comodernistic.
\end{thm}
Stanley defined supersolvable lattices in \citep{Stanley:1972} to
abstract the interesting combinatorics of the subgroup lattices of
supersolvable groups to general lattices. Theorem~\ref{thm:SgLatComodern}
says that comodernism is one possibility for a similar abstraction
for solvable groups. A result of a similar flavor was earlier proved
by Schmidt \citep{Schmidt:1968}; our innovation with comodernism
is to require a lattice-theoretic analogue of a composition series
in every interval of the lattice. We further discuss possible notions
of solvability for lattices in Section~\ref{sec:SolvableSglats}.

Shareshian in \citep{Shareshian:2001} showed that a group $G$ is
solvable if and only if $L(G)$ is shellable. Theorems~\ref{thm:MainThm}
and \ref{thm:SgLatComodern} give a new proof of the ``only if''
direction of this result. For the ``if'' direction, Shareshian needed
a hard classification theorem from finite group theory. Our proof
of Theorem~\ref{thm:SgLatComodern} does not rely on hard classification
theorems. On the other hand, it follows directly from Shareshian's
proof that $G$ is solvable if $L(G)$ is sequentially Cohen-Macaulay.
Thus, Shareshian's Theorem gives a topological characterization of
solvable groups. The characterization given in Theorem~\ref{thm:SgLatComodern}
is lattice-theoretic, rather than topological. It is an interesting
open problem to give a classification-free proof that if $L(G)$ is
sequentially Cohen-Macaulay, then $G$ is solvable.

\subsection{Organization}

This paper is organized as follows. In Section~\ref{sec:Preliminaries}
we recall some of the necessary background material. We prove Theorem~\ref{thm:MainThm}
(our main theorem) in Section~\ref{sec:ProofMainThm}. 

In the remainder of the paper we show how to apply comodernism and
Theorem~\ref{thm:MainThm} to various classes of lattices. The techniques
may be illustrative for those who wish to prove additional classes
of lattices to be comodernistic. In Section~\ref{sec:OrderCongLats},
we examine order congruence lattices, and prove Theorem~\ref{thm:OrdCongShellable}.
In Section~\ref{sec:SolvableSglats} we prove Theorem~\ref{thm:SgLatComodern},
and argue for comodernism as a notion of solvable for lattices. We
close in Section~\ref{sec:kEqualVariations} by showing that $k$-equal
partition lattices (and variations thereof) are comodernistic. 

\section*{Acknowledgements}

We would like to thank Vincent Pilaud and Michelle Wachs for their
helpful remarks. We also thank the anonymous referee for his or her
thoughtful comments. The example in Figure~\ref{fig:NonshellModbymod}
arose from a question of Hugh Thomas. The second author is grateful
to the University of Miami for their hospitality in the spring of
2016, during which time a part of the paper was written. 

\section{\label{sec:Preliminaries}Preliminaries}

We begin by recalling some necessary background and terminology. Many
readers will be able to skip or skim this section, and refer back
to it as necessary.

\subsection{Posets, lattices, and order complexes}

A poset $P$ is \emph{bounded} if $P$ has a unique least element
$\bot$ and greatest element $\top$. 

Associated to a bounded poset $P$ is the \emph{order complex}, denoted
$\Delta P$, a simplicial complex whose faces consist of all \emph{chains}
(totally ordered subsets) in $P\setminus\{\bot,\top\}$. In particular,
the vertices of $\Delta P$ are the chains of length $0$ in $P\setminus\{\bot,\top\}$,
that is, the elements of $P\setminus\{\bot,\top\}$. The importance
of the order complex in poset theory arises since the Möbius function
$\mu(P)$ (important for inclusion-exclusion) is given by the reduced
Euler characteristic $\tilde{\chi}(\Delta P)$. 

We often say that a bounded poset $P$ possesses a property from simplicial
topology (such as ``shellability''), by which we mean that $\Delta P$
has the same property.

We say that a poset $P$ is \emph{Hasse-connected} if the Hasse diagram
of $P$ is connected as a graph. That is, $P$ is Hasse-connected
if and only for any $x,y\in P$, there is a sequence $x=x_{0},x_{1},\dots,x_{k}=y$
such that $x_{i}$ is comparable to $x_{i+1}$ for each $i$. 

A poset $L$ is a \emph{lattice} if every two elements $x,y\in L$
have a unique greatest lower bound (the \emph{meet} $x\wedge y$)
and unique least upper bound (the \emph{join} $x\vee y$). It is obvious
that every lattice is bounded, hence has an order complex $\Delta L$. 

A poset is \emph{graded} if all its maximal chains have the same length,
where the \emph{length} of a chain is one less than its cardinality.
The \emph{height }of a bounded poset $P$ is the length of the longest
chain in $P$, and the \emph{height }of an element $x$ is the height
of the interval $[\bot,x]$. An \emph{atom} of a bounded poset $P$
is an element of height $1$.

The \emph{order dual} of a poset $P$ is the poset $P^{*}$ with reversed
order relation, so that $x<^{*}y$ in $P^{*}$ exactly when $x>y$
in $P$. Poset definitions may be applied to the dual by prepending
a ``co'': for example, an element $x$ is a \emph{coatom} if $x$
is an atom in $P^{*}$. 

For more background on poset and lattice theory from a general perspective,
we refer to \citep{Stanley:2012}. For more on order complexes and
poset topology, we refer to \citep{Wachs:2007}.

\subsection{Simplicial complexes and shellings}

We assume basic familiarity with homology and cohomology, as exposited
in e.g. \citep{Hatcher:2002,Munkres:1984}. 

A \emph{shelling} of a simplicial complex $\Delta$ is an ordering
of the facets of $\Delta$ that obeys certain conditions, the precise
details of which will not be important to us. Not every simplicial
complex has a shelling; those that do are called \emph{shellable}.

We remark that in the early history of the subject, shellings were
defined only for balls and spheres \citep{Sanderson:1957}. Later,
shellings were considered only for \emph{pure} complexes, that is,
complexes all of whose facets have the same dimension. Nowadays, shellings
are studied on arbitrary simplicial complexes \citep{Bjorner/Wachs:1996}.

Shellable complexes are useful for showing a complex to satisfy the
\emph{Cohen-Macaulay} (in the pure case) or \emph{sequentially Cohen-Macaulay}
property (more generally). These properties are important in commutative
algebra as well as combinatorics.

We refer to \citep{Stanley:1996} for more on shellable and Cohen-Macaulay
complexes.

\subsection{\label{subsec:Prelims-CL}$CL$-labelings and $EL$-labelings}

The definition of a shelling is often somewhat unwieldy to work with
directly, and it is desirable to find tools through which to work.
One such tool is given by a $CL$-labeling, which we will now define. 

If $x$ and $y$ are elements in a poset $P$, we say that $y$ \emph{covers}
$x$ when $x<y$ but there is no $z\in P$ so that $x<z<y$. In this
situation, we write $x\lessdot y$, and may also say that $x\lessdot y$
is a \emph{cover relation.} Thus, a cover relation is an edge in the
Hasse diagram of $P$. A \emph{rooted cover relation} is a cover relation
$x\lessdot y$ together with a maximal chain from $\bot$ to $x$
(called the \emph{root}). 

A \emph{rooted interval} is an interval $[x,y]$ together with a maximal
chain $\mathbf{r}$ from $\hat{0}$ to $x$. In this situation, we
use the notation $[x,y]_{\mathbf{r}}$. Notice that every atomic cover
relation of $[x,y]_{\mathbf{r}}$ can be rooted by $\mathbf{r}$. 

A \emph{chain-edge} \emph{labeling} of a bounded poset $P$ is a function
$\lambda$ that assigns an element of an ordered set (which will always
for us be $\zz$) to each rooted cover relation of $P$. Then $\lambda$
assigns a word over $\zz$ to each maximal chain on any rooted interval
by reading the cover relation labels in order, so e.g. the word associated
with $\bot\lessdot x_{1}\lessdot x_{2}\lessdot x_{3}\lessdot\dots$
is $\lambda(\bot\lessdot x_{1},\hat{0})\lambda(x_{1}\lessdot x_{2},\hat{0}\lessdot x_{1})\lambda(x_{2}\lessdot x_{3},\hat{0}\lessdot x_{1}\lessdot x_{2})\cdots$.
\begin{rem}
Since many researchers may be less familiar with $CL$-labelings and
the machinery behind them, it may be helpful to think of a chain-edge
labeling via the following dynamical process. Begin at $\bot$, and
walk up the maximal chain $\bot=x_{0}\lessdot x_{1}\lessdot x_{2}\lessdot\dots\lessdot\top$.
At each step $i$, assign a label to the label $x_{i-1}\lessdot x_{i}$.
In assigning the label, you are allowed to look backwards at where
you have been, but are \uline{not} allowed to look forwards at
where you may go. At each step, you add the assigned label to the
end of a word associated with the maximal chain. 
\end{rem}
We say that a maximal chain $\mathbf{c}$ is \emph{increasing} if
the word associated with $\mathbf{c}$ is strictly increasing, and
\emph{decreasing} if the word is weakly decreasing. We order maximal
chains by the lexicographic order on the associated words.
\begin{defn}
A\emph{ $CL$-labeling} is a chain-edge labeling that satisfies the
following two conditions on each rooted interval $[x,y]_{\mathbf{r}}$:

\begin{enumerate}
\item There is a unique increasing maximal chain $\mathbf{m}$ on $[x,y]_{\mathbf{r}}$,
and
\item the increasing chain $\mathbf{m}$ is strictly earlier in the lexicographic
order than any other maximal chain on $[x,y]_{\mathbf{r}}$.
\end{enumerate}
\end{defn}
If a $CL$-labeling $\lambda$ assigns the same value to every $x\lessdot y$
irrespective of the choice of root, then we say $\lambda$ is an \emph{$EL$-labeling}. 

Björner \citep{Bjorner:1980} and Björner and Wachs \citep{Bjorner/Wachs:1983,Bjorner/Wachs:1996}
introduced $CL$-labelings, and proved the following theorem.
\begin{thm}
\citep[Theorem 5.8]{Bjorner/Wachs:1997} If $\lambda$ is a $CL$-labeling
of the bounded poset $P$, then the lexicographic order on the maximal
chains of $P$ is a shelling order of $\Delta P$. In this case, a
cohomology basis for $\Delta P$ is given by the decreasing maximal
chains of $P$, and $\Delta P$ is homotopy equivalent to a bouquet
of spheres in bijective correspondence with the decreasing maximal
chains.
\end{thm}
For this reason, bounded posets with a $CL$- or $EL$-labeling are
often called \emph{$CL$- or $EL$-shellable}. Since the order complex
of $P$ and that of the order dual of $P$ coincide, either a $CL$-labeling
or a dual $CL$-labeling implies shellability of $\Delta P$.

From the cohomology basis, it is straightforward to compute Euler
characteristic, hence also Möbius number.
\begin{cor}
\citep[Proposition 5.7]{Bjorner/Wachs:1997} If $\lambda$ is a $CL$-labeling
of the bounded poset $P$, then the Möbius number of $P$ is given
by 
\begin{align*}
\mu(P)=\tilde{\chi}(\Delta P)= & \,\#\mbox{even length decreasing maximal chains in }P\\
 & -\#\mbox{odd length decreasing maximal chains in }P.
\end{align*}
\end{cor}

\subsection{\label{subsec:OrderCongBackground}Order congruence lattices}

If $P$ and $Q$ are posets, then a map $\varphi:P\rightarrow Q$
is \emph{order-preserving} if whenever $x\leq y$, it also holds that
$\varphi(x)\leq\varphi(y)$. The \emph{level set partition} of a map
$\varphi:P\rightarrow Q$ is the partition with blocks of the form
$\varphi^{-1}(q)$. If $\pi$ is the level set partition of an order
preserving map $\varphi:P\rightarrow Q$, then $\pi$ is an \emph{order
partition} of $P$. Since every poset has a linear extension, it is
easy to see that it would be equivalent to restrict the definition
of order partition to the case where $Q=\mathbb{Z}$.

As previously defined, the order congruence lattice $\ordcong(P)$
is the subposet of the partition lattice $\Pi_{P}$ consisting of
all order partitions of $P$. The cover relations in $\ordcong(P)$
correspond to merging blocks in an order partition, subject to a certain
compatibility condition.
\begin{example}
\label{exa:Ordcong3chain} Consider $P=[3]$ with the usual order.
Then the function mapping $1,2$ to $1$ and $3$ to $2$ is order-preserving,
so $12\,\vert\,3\in\ordcong([3]).$ Similarly, the partition $1\,\vert\,23\in\ordcong([3])$.
It is not difficult to see, however, that there is no order-preserving
map with level set partition $13\,\vert\,2$. Thus, the lattice $\ordcong([3])$
is isomorphic to the Boolean lattice on 2 elements.

More generally, an elementary argument shows that the order congruence
lattice of a chain on $n$ elements is isomorphic to a Boolean lattice
on $n-1$ elements. It is obvious that the order congruence lattice
of an antichain is the usual partition lattice. Thus, order congruence
lattices interpolate between Boolean lattices and partition lattices. 
\end{example}
It is easy to confuse $\ordcong(P)$ with another closely related
lattice defined on a poset $P$. We say that a subset $S\subseteq P$
is \emph{order convex} if whenever $a\leq b\leq c$ with $a,c\in S$,
then also $b\in S$. The \emph{order convexity partition lattice}
of $P$, denoted $\ordconv(P)$, consists of all partitions where
every block is order convex. There is some related literature on the
related lattice of all order convex subsets of a poset, going back
to \citep{Birkhoff/Bennett:1985}. 

We do not know if order convexity lattices must always be comodernistic
or shellable, as intervals of the form $[\pi,\top]$ in $\ordconv$
seem difficult to describe.
\begin{example}
Consider the \emph{bowtie poset} $B$, with elements $a_{1},a_{2},b_{1},b_{2}$
and relations $a_{i}<b_{j}$ (for $i,j\in\{1,2\}$). As $B$ has height
1, all subsets are order convex, so that $\ordconv(B)\cong\Pi_{4}$.
However, the partitions $a_{1}b_{1}\,\vert\,a_{2}b_{2}$ and $a_{1}b_{2}\,\vert\,a_{2}b_{1}$
are not order congruence partitions, so are not in $\ordcong(B)$.
\end{example}
We additionally caution the reader that the notion of order congruence
considered here is less restrictive than that considered in \citep{Reading:2004},
where congruences are required to respect lower/upper bounds.

In another point of view, it is straightforward to show that order-preserving
partitions are in bijective correspondence with certain quotient objects
of $P$. Thus, the order congruence lattice assigns a lattice structure
to quotients of $P$. See \citep[Section 3]{Jenca/Sarkoci:2014} for
more on the quotient view of $\ordcong(P)$.

For our purposes, it will be enough to understand intervals above
atoms in $\ordcong(P)$. Say that elements $x,y$ of poset $P$ are
\emph{compatible} if either $x\lessdot y$, $y\lessdot x$, or $x,y$
are incomparable. If $x,y$ are compatible in $P$, then $P_{x\sim y}$
is the poset obtained by identifying $x$ and $y$. That is, $P_{x\sim y}$
is obtained from $P$ by replacing $x,y$ with $w$, subject to the
relations $z<w$ whenever $z<x$ or $z<y$, and $z>w$ whenever $z>x$
or $z>y$. We remark in passing that this identification is an easy
special case of the quotienting viewpoint discussed above.
\begin{lem}
\label{lem:OrdCongAtoms}Let $P$ be a poset. A partition $\pi$ of
$P$ is an atom of $\ordcong(P)$ if and only if $\pi$ has exactly
one non-singleton block consisting of compatible elements $\{x,y\}$.
In this situation, we have the lattice isomorphism $[\pi,\top]_{\ordcong(P)}\cong\ordcong(P_{x\sim y})$.
\end{lem}
Repeated application of Lemma~\ref{lem:OrdCongAtoms} allows us to
understand any interval of the form $[\pi,\top]$ in $\ordcong(P)$. 

Although we will not use need this, intervals of the form $[\bot,\pi]$
are also not difficult to understand. Let $\pi$ have blocks $B_{1},B_{2},\dots,B_{k}$.
It is well-known (see e.g. \citep[Example 3.10.4]{Stanley:2012})
that an interval of this form in the full partition lattice is isomorphic
to the product of smaller partition lattices $\Pi_{B_{1}}\times\cdots\times\Pi_{B_{k}}$.
It is straightforward to see via the order-preserving mapping $P\rightarrow\mathbb{Z}$
definition that a similar result holds in the order congruence lattice.
That is, $[\bot,\pi]$ in $\ordcong(P)$ is lattice-isomorphic to
$\ordcong(B_{1})\times\cdots\times\ordcong(B_{k})$, where $B_{i}$
refers to the induced subposet on $B_{i}\subseteq P$. Combining this
observation with Lemma~\ref{lem:OrdCongAtoms} allows us to write
any interval in $\ordcong(P)$ as a product of order congruence lattices
of quotients of subposets. We find it simpler to give more direct
arguments, but readers familiar with poset products may appreciate
this connection.

\subsection{\label{subsec:Prelims-lm}Supersolvable and semimodular lattices}

We previously defined an element $m$ of a lattice $L$ to be left-modular
if $(x\vee m)\wedge y=x\vee(m\wedge y)$ for all pairs $x<y$. 

A lattice is \emph{modular} if every element is left-modular. A lattice
$L$ is usually defined to be \emph{semimodular} if whenever $a\wedge b\lessdot a$
in $L$, then $b\lessdot a\vee b$. We prefer the following equivalent
definition, which highlights the close connection between semimodularity
and comodernism:
\begin{lem}[{see e.g. \citep[essentially Theorem 1.7.2]{Stern:1999}}]
 \label{lem:AltDefSemimod} A lattice $L$ is semimodular if and
only for every interval $[x,y]$ of $L$, every atom of $[x,y]$ is
left-modular (as an element of $[x,y]$).
\end{lem}
Thus, the definition of a modernistic lattice is obtained from that
of a semimodular lattice by weakening a single universal quantifier
to an existential quantifier.

An \emph{$M$-chain} in a lattice is a maximal chain consisting of
left-modular elements. A lattice is \emph{left-modular} if it has
an $M$-chain, and \emph{supersolvable} if it is graded and left-modular. 

Supersolvable lattices were originally defined by Stanley \citep{Stanley:1972},
in a somewhat different form. The theory of left-modular lattices
was developed in a series of papers \citep{Blass/Sagan:1997,Liu:1999,Liu/Sagan:2000,McNamara/Thomas:2006},
and it was only in \citep{McNamara/Thomas:2006} that it was noticed
that Stanley's original definition of supersolvable is equivalent
to graded and left-modular. 

There is an explicit cohomology basis for a supersolvable lattice,
which does not seem to be as well-known as is deserved. A \emph{chain
of complements} to an $M$-chain $\mathbf{m}=\{\bot=m_{0}\lessdot m_{1}\lessdot\dots\lessdot m_{n}=\top\}$
is a chain of elements $\mathbf{c}=\{\bot=c_{n}\lessdot c_{n-1}\lessdot\dots\lessdot c_{0}=\top\}$
so that each $c_{i}$ is a \emph{complement} to $m_{i}$, that is,
so that $c_{i}\vee m_{i}=\top$ and $c_{i}\wedge m_{i}=\bot$. A less
explicit form of the following appears in \citep{Bjorner:1980,Stanley:1972},
and a special case in \citep{Thevenaz:1985}.
\begin{thm}
\label{thm:SScohomology} If $L$ is a supersolvable lattice with
a fixed $M$-chain $\mathbf{m}$, then a cohomology basis for $\Delta L$
is given by the chains of complements to $\mathbf{m}$. In particular,
the Möbius number of $L$ is (up to sign) the number of such chains.
\end{thm}
A (strong form of a) homology basis for supersolvable lattices appears
in \citep{Schweig:2009}.

\subsection{Left-modularity}

We now recall some additional basic properties of left-modular elements.
First, we state more carefully the equivalent ``no pentagon'' condition
mentioned in the Introduction.
\begin{lem}
\label{lem:NoPentagons} \citep[Proposition 1.5]{Liu/Sagan:2000}
An element $m$ of the lattice $L$ is left-modular if and only if
for every $a<c$ in $L$, we have $a\wedge m\neq c\wedge m$ or $a\vee m\neq c\vee m$.
\end{lem}
The pentagon lattice (usually notated as $N_{5}$) consists of elements
$\bot,\top,a,b,c$ with the only nontrivial relation being $a<c$.
Lemma~\ref{lem:NoPentagons} says exactly that $m$ is left-modular
if and only if $m$ never plays the role of $b$ in a sublattice of
$L$ isomorphic to $N_{5}$. Thus, Lemma~\ref{lem:NoPentagons} is
a pleasant generalization of the characterization of modular lattices
as those with no pentagon sublattices.

Another useful fact is:
\begin{lem}
\label{lem:LMProjection} \citep[Proposition 2.1.5]{Liu:1999} If
$m$ is a left-modular element of the lattice $L$, and $x<y$ in
$L$, then $x\vee m\wedge y$ is a left-modular element of the interval
$[x,y]$.
\end{lem}
Finally, we give an alternate characterization of left-modularity
of coatoms, in the flavor of Lemma~\ref{lem:AltDefSemimod}. This
characterization will be useful for us in the proof of Theorems~\ref{thm:MainThm}
and \ref{thm:CLlabeling}, and is also often easy to check.
\begin{lem}[Left-modular Coatom Criterion]
 \label{lem:LMcoatomCond} Let $m$ be a coatom of the lattice $L$.
Then $m$ is left-modular in $L$ if and only if for every $y$ such
that $y\not\leq m$ we have $m\wedge y\lessdot y$.
\end{lem}
\begin{proof}
If $m\wedge y<z<y$, then $z<y$ violate the condition of  Lemma~\ref{lem:NoPentagons}
with $m$, hence $m$ is not left-modular. Conversely, if $z<y$ violate
the condition of Lemma~\ref{lem:NoPentagons}, then $y\not\leq m$,
and $m\wedge y=m\wedge z<z<y$.
\end{proof}
\begin{cor}
\label{cor:ModProjection} If $L$ is a lattice and $m$ a left-modular
coatom of $L$, then for any $x<y$ in $L$ either $x\vee m\wedge y=y$
or else $x\vee m\wedge y\lessdot y$.
\end{cor}
\begin{proof}
If $x\not\leq m$ or $y\leq m$, then $x\vee m\wedge y=y$. Otherwise,
apply Lemmas~\ref{lem:LMProjection} and \ref{lem:LMcoatomCond}.
\end{proof}

\subsection{\label{subsec:PrelimGrps}Group theory}

We recall that a group $G$ is said to be \emph{solvable} if either
of the following equivalent conditions is met:
\begin{enumerate}
\item \label{enu:SolvChiefseries}There is a chain $1=N_{0}\subset N_{1}\subset N_{2}\subset\dots\subset N_{k}=G$
of subgroups in $G$, so that each $N_{i}$ is normal in $G$, and
so that each factor $N_{i}/N_{i-1}$ is abelian.
\item \label{enu:SolvCompseries}There is a chain $1=H_{0}\subsetdot H_{1}\subsetdot H_{2}\subsetdot\dots\subsetdot H_{n}=G$
of subgroups in $G$, so that each $H_{i}$ is normal in $H_{i+1}$
(but is not necessarily normal in $G$). Note that it follows in this
case that each factor $H_{i}/H_{i-1}$ is cyclic of prime order.
\end{enumerate}
Since every subgroup of a solvable group is solvable, an alternative
form of the latter is:
\begin{enumerate}
\item [(2')] For every subgroup $H\subseteq G$, there is a subgroup $K\subsetdot H$
such that $K\normalin H$.
\end{enumerate}
A subgroup $H$ of $G$ is said to be \emph{subnormal} if there is
a chain $H\normalin L_{1}\normalin L_{2}\normalin\dots\normalin G$.
Thus, Condition~(\ref{enu:SolvCompseries}) says a group is solvable
if and only if $G$ has a maximal chain consisting of subnormal subgroups. 

A group is \emph{supersolvable} if there is a maximal chain in $L(G)$
consisting of subgroups normal in $G$. Thus, a group is supersolvable
if there is a chain which simultaneously meets the conditions in (1)
and (2). One important fact about the subgroup lattice of supersolvable
groups is:
\begin{thm}
\citep{Iwasawa:1941} For a group $G$, the subgroup lattice $L(G)$
is graded if and only if $G$ is supersolvable.
\end{thm}
Subgroup lattices were one of the motivations for early lattice theorists
in making the definition of (left-)modularity. It follows easily from
the Dedekind Identity (see Lemma~\ref{lem:PermutableProps}) that
if $N\normalin G$, then $N$ is left-modular in $L(G)$. In particular,
if $G$ is a supersolvable group, then $L(G)$ is a supersolvable
lattice.

Moreover, a normal subgroup $N$ satisfies a stronger condition. An
element $m$ of a lattice is said to be \emph{modular} (or \emph{two-sided-modular})
if it neither plays the role of $b$ nor of $a$ in any pentagon sublattice,
where $a,b$ are as in the discussion following Lemma~\ref{lem:NoPentagons}.
A second application of the Dedekind Identity shows any normal subgroup
to be modular in $L(G)$. 

The following lemma, whose proof is immediate from the definitions,
says that left-modularity and two-sided-modularity are essentially
the same for the purpose of comodernism arguments.
\begin{lem}
\label{lem:LmMaxlIsMod}If $L$ is a lattice, and $m$ is a maximal
left-modular element, then $m$ is modular.
\end{lem}
We refer to e.g. \citep[Chapter A]{Doerk/Hawkes:1992} for further
general background on group theory, and to \citep{Schmidt:1994} for
the reader interested in further background on lattices of subgroups.

\section{\label{sec:ProofMainThm}Proof of Theorem~\ref{thm:MainThm}}

\subsection{sub-$M$-chains}

As discussed in Section~\ref{subsec:Prelims-lm}, a lattice is left-modular
if it has an $M$-chain, that is, a maximal chain consisting of left-modular
elements. The reader may be reminded of the maximal chain consisting
of normal elements in the definition of a supersolvable group. 

We extend the notion of $M$-chain to comodernistic lattices. A maximal
chain $\bot=m_{0}\lessdot m_{1}\lessdot\dots\lessdot m_{n}=\top$
in $L$ is a \emph{sub-$M$-chain} if for every $i$, the element
$m_{i}$ is left-modular in the interval $[\bot,m_{i+1}]$. The reader
may be reminded of the maximal subnormal chain in a solvable group.
It is straightforward to show that a lattice is comodernistic if and
only if every interval has a sub-$M$-chain.

Stanley \citep{Stanley:1974} and Björner \citep{Bjorner:1980} showed
that any supersolvable lattice has an $EL$-labeling, and Liu \citep{Liu:1999}
extended this to any left-modular lattice. If $\mathbf{m}^{(ss)}=\left\{ \bot=m_{0}^{(ss)}\lessdot m_{1}^{(ss)}\lessdot\dots\lessdot m_{n}^{(ss)}=\top\right\} $
is an $M$-chain, then the $EL$-labeling is defined as follows:
\begin{align}
\lambda_{ss}(x\lessdot y) & =\max\{i\,:\,x\vee m_{i-1}^{(ss)}\wedge y=x\}\label{eq:ssELlabeling}\\
 & =\min\{i\,:\,x\vee m_{i}^{(ss)}\wedge y=y\}.\nonumber 
\end{align}

The essential observation involved in proving Theorem~\ref{thm:MainThm}
is that, if we replace the $M$-chain used for $\lambda_{ss}$ with
a sub-$M$-chain, then we can still label the atomic cover relations
of $L$ in the same manner as in $\lambda_{ss}$. More precisely,
if $\mathbf{m}$ is a sub-$M$-chain in a lattice $L$, then let 
\begin{align}
\lambda(\bot\lessdot a) & =1+\max\{i\,:\,m_{i}\wedge a=\hat{0}\}.\label{eq:atomlabeling}
\end{align}
 Adding $1$ is not essential, and we do so only so that the labels
will be in the range $1$ through $n$, rather than $0$ through $n-1$.

\subsection{The $CL$-labeling}

We construct the full $CL$-labeling recursively from (\ref{eq:atomlabeling}).

We say that a chain $\mathbf{c}$ is \emph{indexed} by a subset $S=\{i_{1}<\cdots<i_{k}\}$
of the integers if $\mathbf{c}=\{c_{i_{1}}<\cdots<c_{i_{k}}\}$. That
is, we associate an index or label with each element of the chain.
Notice that we require the indices to (strictly) respect order.

We will need the following somewhat-technical lemma to handle non-graded
lattices. 
\begin{lem}
\label{lem:EnoughModElts}Let $L$ be a lattice with a sub-$M$-chain
$\mathbf{m}$ of length $n$. Then no chain of $L$ has length greater
than $n$.
\end{lem}
\begin{proof}
We proceed by induction on $n$. The base case is trivial. Suppose
that $\hat{0}=c_{0}\lessdot c_{1}\lessdot\cdots\lessdot c_{\ell}=\hat{1}$
is some chain in $L$. Let $m_{n-1}\lessdot1$ be the unique coatom
in $\mathbf{m}$, and $i$ be the greatest index such that $c_{i}\leq m_{n-1}$.
Then $c_{j}\vee m_{n-1}=\top$ for any $j>i$, so by the left-modular
property 
\[
c_{i+1}\wedge m_{n-1}<c_{i+2}\wedge m_{n-1}<\cdots<c_{\ell}\wedge m_{n-1}=m_{n-1}.
\]
Thus, 
\[
\bot=c_{0}\lessdot c_{1}\lessdot\cdots\lessdot c_{i}=c_{i+1}\wedge m_{n-1}<c_{i+2}\wedge m_{n-1}<\cdots<c_{\ell}\wedge m_{n-1}=m_{n-1}
\]
 is a chain of length $\ell-1$ on $[\bot,m_{n-1}]$, and by induction
$\ell-1\leq n-1$. 
\end{proof}
\begin{defn}
\label{def:comodernLabeling}Let $L$ be a comodernistic lattice of
height $n$. Take a fixed sub-$M$-chain $\mathbf{m}$ given as $\bot=m_{0}\lessdot m_{1}\lessdot\cdots\lessdot m_{n}=\top$
as the starting point for a recursive construction. 

Let $x\lessdot a,\mathbf{r}$ be a rooted cover relation. Assume by
recursion that we are given a sub-$M$-chain $\mathbf{m}^{(\mathbf{r})}$
on $[x,\top]$. Further assume that the elements of $\mathbf{m}^{(\mathbf{r})}$
are indexed by a subset $S\subseteq[n]\cup\left\{ 0\right\} $, and
that $\top=m_{n}^{(\mathbf{r})}$. Label $x\lessdot a$ as in (\ref{eq:atomlabeling})
that is, as 
\begin{equation}
\lambda(x\lessdot a)=1+\max\{i\,:\,m_{i}^{(\mathbf{r})}\wedge a=x\}.\label{eq:comodCLlabeling}
\end{equation}
To continue the recursion, it remains to construct an indexed sub-$M$-chain
$\mathbf{m}^{(\mathbf{r}\cup a)}$ on $[a,\top]$.

Suppose that $\lambda(x\lessdot a)=1+i$. It is clear that $m_{i}^{(\mathbf{r})}$
is the greatest element of $\mathbf{m}^{(\mathbf{r})}$ such that
$a\not\leq m_{i}^{(\mathbf{r})}$. By abuse of notation, let $\mathbf{m}_{>i}^{(\mathbf{r})}$
be the portion of $\mathbf{m}$ that is greater than $m_{i}^{(\mathbf{r})}$,
and let $S_{>i}$ be the indices greater than $i$ on $\mathbf{m}^{(\mathbf{r})}$.
Thus, the labels of $\mathbf{m}_{>i}^{(\mathbf{r})}$ are exactly
$S_{>i}$. Let $S_{<i}=S\setminus(S_{>i}\cup i)$ similarly be the
indices less than $i$ on $\mathbf{m}^{(\mathbf{r})}$. 

Now by construction, all elements of $\mathbf{m}_{>i}^{(\mathbf{r})}$
are greater than $a$. By the comodernistic property, the submodular
chain $\mathbf{m}_{>i}^{(\mathbf{r})}$ may be completed to a sub-$M$-chain
$\mathbf{m}^{(\mathbf{r}\cup a)}$ for $[a,\top]$. Preserve the indices
on $\mathbf{m}_{>i}^{(\mathbf{r})}$, and index the elements of $\mathbf{m}^{(\mathbf{r}\cup a)}\setminus\mathbf{m}^{(\mathbf{r})}$
by elements of $S_{<i}$. It follows by applying Lemma~\ref{lem:EnoughModElts}
on $[\bot,m_{i+1}]$ that there are enough indices available in $S_{<i}$
to perform such indexing.

The recursion can now continue, which completes the definition of
the $CL$-labeling.
\end{defn}
\begin{notation}
\label{nota:Labeling}Throughout the remainder of Section~\ref{sec:ProofMainThm},
we fix $L$ to be a comodernistic lattice of height $n$, with a sub-$M$-chain
$\mathbf{m}=\{\bot=m_{0}\lessdot m_{1}\lessdot\cdots\lessdot m_{n}=\top\}$.
Indeed, we select a sub-$M$-chain on every interval, which uniquely
determines a chain-edge labeling $\lambda$ as in Definition~\ref{def:comodernLabeling}.
\end{notation}
\begin{rem}
By repeated application of Lemma~\ref{lem:LMProjection}, it follows
that if $L$ has an $M$-chain $\mathbf{m}=\{\bot=m_{0}\lessdot m_{1}\lessdot\cdots\lessdot m_{n}=\top\}$,
then the set $\left\{ u\vee m_{i}\wedge v\right\} $ is an $M$-chain
for the interval $[u,v]$. With this choice of (sub)-$M$-chain on
each interval, the labeling $\lambda$ of Notation~\ref{nota:Labeling}
coincides with $\lambda_{ss}$.
\end{rem}
We are now ready to prove the following refinement of Theorem~\ref{thm:MainThm}. 
\begin{thm}
\label{thm:CLlabeling}The labeling $\lambda$ of Notation~\ref{nota:Labeling}
is a $CL$-labeling.
\end{thm}
\begin{proof}
It is clear from construction that $\lambda$ is a chain-edge labeling.
By the recursive construction, it suffices to show that an interval
of the form $[\bot,y]$ has a unique increasing maximal chain, and
that every lexicographically first maximal chain on $[\bot,y]$ is
increasing.

Let $\mathbf{m}=\{\bot=m_{0}\lessdot m_{1}\lessdot\cdots\lessdot m_{n}=\top\}$
be the sub-$M$-chain used to define the labeling. Let $\ell=1+\max\{i\,:\,m_{i}\wedge y<y\}$.
It is clear from the construction that every atomic cover relation
on $[\bot,y]$ receives a label that is at most $\ell$. Since the
elements greater than $m_{\ell-1}$ are preserved until the corresponding
labels are used, no chain on $[\bot,y]$ receives any label greater
than $\ell$. 

Similarly, the $m_{\ell-1}$ element in the sub-$M$-chain is by construction
preserved until the $\ell$ label is used, and a chain receives an
$\ell$ label when it leaves the interval $[\bot,m_{\ell-1}]$. Since
$y\notin[\bot,m_{\ell-1}]$, we see that every maximal chain on $[\bot,y]$
receives an $\ell$ label. Thus, an increasing chain must have the
$\ell$ label on its last cover relation. 

But Corollary~\ref{cor:ModProjection} gives $m_{\ell-1}\wedge y$
to be a coatom of $[\bot,y]$. It follows by the definition of the
labeling that every increasing chain on $[\bot,y]$ must end with
$m_{\ell-1}\wedge y\lessdot y$. An easy induction now yields the
only increasing chain to be $\hat{0}=m_{0}\wedge y\leq m_{1}\wedge y\leq\dots\leq m_{\ell-1}\wedge y$,
the ``projection'' of the sub-$M$-chain to $[\bot,y]$. As $m_{\ell-1}\wedge y\lessdot y$,
the projection chain is in particular maximal.

We now show that this chain is the unique lexicographically first
chain. In the construction, the least label of an atomic cover relation
on $[\bot,y]$ corresponds with the least $m_{i+1}$ such that $m_{i+1}\wedge y>\hat{0}$.
But this is the (unique) first non-$\bot$ element of the increasing
chain. The desired follows.
\end{proof}

\subsection{More details about the $CL$-labeling}

Any chain-edge labeling assigns a word to each maximal chain of $L$.
Since when we label a cover relation with $i$ according to $\lambda$,
we remove $i$ from the index set (used for available labels), we
obtain a result extending one direction of \citep[Theorem 1]{McNamara:2003}.
\begin{lem}
The chain-edge labeling $\lambda$ assigns a word with no repeated
labels to each maximal chain in $L$. 

Thus, if $L$ is graded of height $n$, then $\lambda$ assigns a
permutation in $S_{n}$ to each maximal chain.
\end{lem}
The decreasing chains are also easy to (recursively) understand. Recall
that if $x$ and $y$ are lattice elements, then $x$ is a \emph{complement}
to $y$ if $x\vee y=\top$ and $x\wedge y=\bot$. The following is
an extension of Theorem~\ref{thm:SScohomology} for comodernistic
lattices.
\begin{lem}
\label{lem:ComputingMobius}If $\bot\lessdot c_{1}\lessdot\cdots\lessdot\top$
is a decreasing chain of $L$ with respect to $\lambda$, then $c_{1}$
is a complement to $m_{n-1}$.
\end{lem}
\begin{proof}
As in the proof of Theorem~\ref{thm:CLlabeling}, every maximal chain
on $[\bot,\top]$ contains an $n$ label. Thus $\lambda(\bot\lessdot c_{1})=n$,
so $m_{n-1}\wedge c_{1}=\bot$. The result follows.
\end{proof}
\begin{figure}
\includegraphics[scale=0.8]{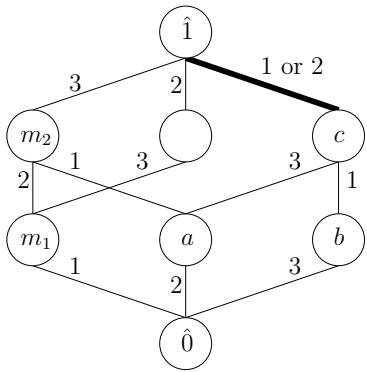}\caption{\label{fig:ComodernNotSS}A comodernistic labeling of a lattice}
\end{figure}
It is natural to ask whether theorems about supersolvable geometric
lattices (see e.g. \citep{Stanley:1972}) extend to comodernistic
geometric lattices. The answer to this question is positive, but for
rather uninteresting reasons:
\begin{prop}
\label{prop:Geometric+Comod} If $L$ is a geometric lattice, then
any sub-$M$-chain of $L$ is also an $M$-chain. Thus, a lattice
$L$ is geometric and comodernistic if and only if $L$ is geometric
and supersolvable.
\end{prop}
\begin{proof}
A result of Brylawski \citep[Proposition 3.5]{Brylawski:1975} says
that if $m$ is modular in a geometric lattice $L$, and $x$ is modular
in $[\bot,m]$, then $x$ is also modular in $L$. The result now
follows from Lemma~\ref{lem:LmMaxlIsMod} and an inductive argument.
\end{proof}
In contrast to Proposition~\ref{prop:Geometric+Comod}, lattices
that are not geometric may have sub-$M$-chains which are not $M$-chains.
We close this section by working out a small example in detail. 
\begin{example}
We consider the lattice $C$ in Figure~\ref{fig:ComodernNotSS},
which we obtained by removing a single cover relation from the Boolean
lattice on 3 elements. It is easy to check that $m_{2}$ is modular
in $C$, but that $m_{1}$ is not modular in $C$. (Indeed, $m_{1}$
together with $b<c$ generate a pentagon sublattice.) Since any lattice
of height at most 2 is modular, the chain $\bot\lessdot m_{1}\lessdot m_{2}\lessdot\top$
is a sub-$M$-chain, though not an $M$-chain.

With the exception of $c\lessdot\top$, the label of every cover relation
in $C$ is independent of the choice of root. We have indicated these
labels in the diagram. But we notice that the interval $[a,\top]$
inherits the sub-$M$-chain $a\lessdot m_{2}\lessdot\top$, while
the interval $[b,\top]$ has unique maximal (sub-$M$-)chain $b\lessdot c\lessdot1$.
Thus, the edge $c\lessdot\top$ receives a label of $1$ with respect
to root $\bot\lessdot a\lessdot c$, but a label of $2$ with respect
to root $\bot\lessdot b\lessdot c$.

The reader may have noticed that the atom $a$ is indeed left-modular.
Thus, although we have shown the comodernistic labeling determined
by the given sub-$M$-chain, there is also a supersolvable $EL$-labeling
of $C$. We will see in Example~\ref{exa:OrdcongN5} and Figure~\ref{fig:OrdcongN5}
a lattice that is neither geometric nor supersolvable, but that is
comodernistic.
\end{example}

\section{\label{sec:OrderCongLats}Order congruence lattices}

In this section, we examine the order congruence lattices of posets,
as considered in the introduction and in Section~\ref{subsec:OrderCongBackground}.
We prove Theorem~\ref{thm:OrdCongShellable}, and apply Lemma~\ref{lem:ComputingMobius}
to calculate the Möbius number of $\ordcong(P)$. 

\subsection{Order congruence lattices are comodernistic}

A useful tool for showing certain lattices to be comodernistic is
given by the following lemma.
\begin{lem}
\label{lem:LMMeetSubsemi}Let $L$ be a meet subsemilattice of a lattice
$L_{+}$. If $m\in L_{+}$ is a left-modular coatom in $L_{+}$, and
$m\in L$, then $m$ is also left-modular in $L$. 
\end{lem}
\begin{proof}
Since $m$ is a coatom in $L_{+}$ and therefore in $L$, the join
of $x$ and $m$ is either $m$ (if $x\leq m$) or $\hat{1}$ (otherwise)
in both lattices. In particular, the join operations in $L$ and $L_{+}$
agree on $m$, and we already know the meet operations agree by the
subsemilattice condition. The result now follows by Lemma~\ref{lem:NoPentagons}.
\end{proof}
The following theorem follows immediately.
\begin{thm}
\label{thm:LMeltsPartitionLattice} If $L$ is a meet subsemilattice
of the partition lattice $\Pi_{S}$, and $m\in L$ is a partition
$x\,\vert\,(S\setminus x)$ for some element $x\in S$, then $m$
is a left-modular coatom of $L$.
\end{thm}
We now show:
\begin{lem}
\label{lem:OrdCongMeetsemilat}If $P$ is any poset, then the order
congruence lattice $\ordcong(P)$ is a meet subsemilattice of $\Pi_{P}$.
\end{lem}
\begin{proof}
It is clear from definition that $\ordcong(P)$ is a subposet of $\Pi_{P}$.
It suffices to show that if $\pi_{1},\pi_{2}$ are in $\ordcong(P)$,
then their meet $\pi_{1}\wedge\pi_{2}$ also is in $\ordcong(P)$.
Let $f_{1},f_{2}:P\rightarrow\zz$ be such that $\pi_{1}$ and $\pi_{2}$
are the level sets of $f_{1}$ and $f_{2}$. But then the product
map $f_{1}\times f_{2}:P\rightarrow\zz\times\zz$ (where $\mathbb{Z}\times\zz$
is taken with the product order) has the desired level set partition.
\end{proof}
That order congruence lattices are comodernistic now follows easily.
\begin{proof}[Proof (of Theorem~\ref{thm:OrdCongShellable})]

Let $P$ be a poset, and let $x$ be a maximal element of $P$. Assume
by induction that the result holds for all smaller posets. It is straightforward
to see $m=x\,|\,(P\setminus x)$ is the level set partition of an
order preserving map. By Theorem~\ref{thm:LMeltsPartitionLattice}
and Lemma~\ref{lem:OrdCongMeetsemilat}, the element $m$ is a left-modular
coatom on the interval $[\bot,\top]$. Since $[\bot,m]$ is lattice-isomorphic
to $\ordcong(P\setminus x)$, we get by induction that $[\bot,\pi]$
is comodernistic when $\pi<m$. If $\pi$ is incomparable to $m$,
then $\pi\wedge m$ is a left-modular coatom of $[\bot,\pi]$ by Corollary~\ref{cor:ModProjection}.
Finally, repeated application of Lemma~\ref{lem:OrdCongAtoms} and
induction gives that intervals of the form $[\pi',\top]$ are comodernistic.
The result follows for general intervals $[\pi',\pi]$.
\end{proof}

\subsection{The Möbius number of an order congruence lattice}

\global\long\def\compat{\operatorname{Compat}}

We now use the comodernism of the order congruence lattice $\ordcong(P)$
to recover the Möbius number calculation due to Jen\v{c}a and Sarkoci.
Denote by $\compat(x)$ the set of all $y\in P$ that are compatible
with $x$. That is, $\compat(x)$ consists of all $y$ such that either
$y\lessdot x$, $x\lessdot y$, or $y$ is incomparable to $x$.

Our proof is short and simple. 
\begin{thm}
\label{thm:OrdCongRecursion} \citep[Theorem 3.8]{Jenca/Sarkoci:2014}
For any poset $P$ with maximal element $x$, the Möbius function
of the order congruence lattice satisfies the recurrence 
\[
{\displaystyle \mu(\ordcong(P))=\,\,-\!\!\!\!\!\!\sum_{y\in\compat(x)}\mu(\ordcong(P_{x\sim y})}.
\]
\end{thm}
\begin{proof}
By Lemma~\ref{lem:ComputingMobius} together with the proof of Theorem~\ref{thm:OrdCongShellable},
every decreasing chain of $\ordcong(P)$ begins with a complement
to the (left-modular) order partition $x\,|\,P\setminus x$. Such
complements are easily seen to be atoms $a$ whose non-singleton block
is $\{x,y\}$, where $y\in P$ is compatible with $x$. The result
now follows by Lemma~\ref{lem:OrdCongAtoms}.
\end{proof}
Jen\v{c}a and Sarkoci also show in \citep{Jenca/Sarkoci:2014} that
if $P$ is a Hasse-connected poset, then the number of linear extensions
of $P$ satisfies the same recurrence as $\mu(\ordcong(P))$. We give
a short bijective proof of the same, which has the same flavor as
the proof of the main result in \citep{Edelman/Hibi/Stanley:1989}.
Let $\linext(P)$ denote the set of linear extensions of $P$.
\begin{lem}
\label{lem:LERecursion}Let $P$ be a poset and $x$ a maximal element
of $P$. If $x$ is not also minimal, then there is a bijection 
\[
\linext(P)\rightarrow\bigcup_{y\in\compat(x)}\linext(P_{x\sim y}).
\]
\end{lem}
\begin{proof}
Since $x$ is not minimal, it cannot be the first element in any linear
extension $L$ of $P$. If $y$ is the element immediately preceding
$x$ in $L$, then it is clear that $x$ and $y$ are compatible.
Then $L_{x\sim y}$ is a linear extension of $P_{x\sim y}$. 

To show this map is a bijection, we notice that the process is reversible.
If $L$ is a linear extension of $P_{x\sim y}$, replace the element
corresponding to the identification of $x$ and $y$ with $x$ followed
by $y$ to get a linear extension of $P$.
\end{proof}
\begin{cor}
If $P$ is a Hasse-connected poset, then the decreasing chains of
$\ordcong(P)$ are in bijective correspondence with the linear extensions
of $P$. 

In particular, $\left|\mu(\ordcong(P)\right|$ is the number of linear
extensions of $P$, and $\Delta\ordcong(P)$ is homotopy equivalent
to a bouquet of this number of $(\left|P\right|-3)$-dimensional spheres.
\end{cor}
\begin{proof}
Since $P$ is Hasse-connected, a maximal element $x$ cannot also
be minimal. Moreover, if $P$ is Hasse-connected then $P_{x\sim y}$
is also Hasse-connected. The result now follows immediately by Theorem~\ref{thm:OrdCongRecursion}
and Lemma~\ref{lem:LERecursion}.
\end{proof}
In the case where $P$ is not Hasse-connected, a similar approach
can be followed. Indeed, the same argument as in Lemma~\ref{lem:LERecursion}
applies, except that we must discard the linear extensions that begin
with $x$ in each recursive step where they arise. This argument identifies
the decreasing chains of $\ordcong(P)$ with a recursively-defined
subset of the linear extensions of $P$. We do not have a non-recursive
description of this subset. Jen\v{c}a and Sarkoci give a somewhat
different description of $\left|\mu(\ordcong(P))\right|$ for a non-Hasse-connected
poset $P$ in \citep[Theorem 4.5]{Jenca/Sarkoci:2014}.

\subsection{An order congruence lattice that is neither geometric nor supersolvable}

\begin{figure}
\includegraphics[scale=0.75]{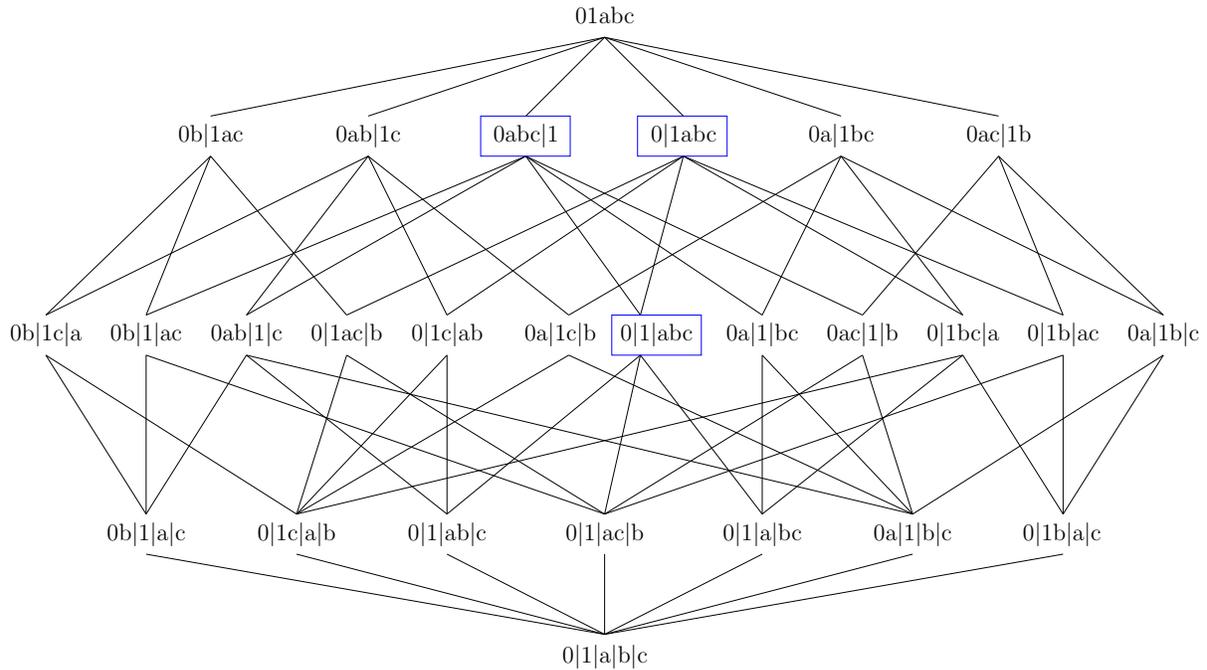}

\caption{\label{fig:OrdcongN5}The order congruence lattice $\protect\ordcong(N_{5})$
of the pentagon lattice $N_{5}$.\protect \\
Nontrivial left-modular elements are shown with rectangles. }
\end{figure}

\begin{example}
\label{exa:OrdcongN5}Consider the pentagon lattice $N_{5}$, obtained
by attaching a bottom and top element $\bot$ and $\top$ to the poset
with elements $a,b,c$ and relation $a<c$. In this case, $\ordcong(N_{5})$
and $\ordconv(N_{5})$ coincide, and are pictured in Figure~\ref{fig:OrdcongN5}.

The reader can verify by inspection that no atom of $\ordcong(N_{5})$
is left-modular, thus, the lattice $\ordcong(N_{5})$ is comodernistic
but neither geometric nor supersolvable. As some of the coatoms of
$\ordcong(N_{5})$ are not left-modular, the dual of $\ordcong(N_{5})$
also fails to be geometric. We remark that order congruence lattices
that fail to be geometric were examined earlier in \citep{Kortesi/Radeleczki/Szilagyi:2005}.
\end{example}

\section{\label{sec:SolvableSglats}Solvable subgroup lattices}

In this section, we discuss applications to and connections with the
subgroup lattice of a group. 

\subsection{Known lattice-theoretic analogues of classes of groups}

Since the early days of the subject, a main motivating object for
lattice theory has been the subgroup lattice of a finite group. Indeed,
a (left-)modular element may be viewed as a purely lattice-theoretic
analogue or extension of a normal subgroup. Focusing on the normal
subgroups characterizing a class of groups then typically gives in
a straightforward way an analogous class of lattices with interesting
properties. For example, every subgroup of an abelian (or more generally
Hamiltonian) group is normal, so a corresponding class of lattices
is that of the modular lattices.

\begin{table}
\begin{tabular}{>{\raggedright}m{3cm}>{\raggedright}p{3.5cm}>{\raggedright}p{2.5cm}l}
Group class~~ &
Lattice class \\
for $L(G)$ &
Characterizes\\
group class? &
Self-dual?\tabularnewline[0.6cm]
\hline 
\noalign{\vskip0.2cm}
cyclic &
distributive &
Yes &
Yes\tabularnewline
\noalign{\vskip0.2cm}
abelian,\\
Hamiltonian &
modular &
No &
Yes\tabularnewline[0.4cm]
nilpotent &
lower semimodular &
No &
No\tabularnewline[0.2cm]
supersolvable &
supersolvable &
Yes &
Yes\tabularnewline[0.2cm]
solvable &
??? &
 &
\tabularnewline
\end{tabular}\bigskip{}

\caption{\label{tab:GroupAndLatClasses} Classes of groups and related classes
of lattices.}
\end{table}

We summarize some of these analogies in Table~\ref{tab:GroupAndLatClasses}. 

We remark that, although every normal subgroup is modular in the subgroup
lattice, not every modular subgroup is normal. Similarly, although
every nilpotent group has lower semimodular subgroup lattice, group
that are not nilpotent may also have lower semimodular subgroup lattice.
For example, the subgroup lattice of the symmetric group on 3 elements
$L(S_{3})$ has height 2, hence is modular (despite being neither
abelian nor nilpotent). As $L(S_{3})$ is lattice isomorphic to $L(\mathbb{Z}_{3}^{2})$,
a subgroup lattice characterization of these classes of groups is
not possible.

It may then be surprising that a group $G$ is supersolvable if and
only if $L(G)$ is a supersolvable lattice. The reason for this is
more superficial than one might hope: Iwasawa proved in \citep{Iwasawa:1941}
that a group is supersolvable if and only if its subgroup lattice
is graded. However, the definition of supersolvable lattice seems
to capture the pleasant combinatorial properties of supersolvable
groups much better than the definition of graded lattice does.

Semimodular and supersolvable lattices have been of great importance
in algebraic and topological combinatorics. In particular, both classes
of lattices are $EL$-shellable, and the $EL$-labeling gives an efficient
method of computing homotopy type, Möbius invariants, etc. 

\subsection{Towards a definition of solvable lattice}

After the previous subsection, it may come as some surprise that there
is no widely-accepted definition of solvable lattice. It is the purpose
of this subsection to make the case for the definition of comodernistic
lattices as one good candidate.

It was independently proved by Suzuki in \citep{Suzuki:1951} and
Zappa in \citep{Zappa:1951} that solvable groups are characterized
by their subgroup lattices. Later Schmidt gave an explicit characterization:
\begin{prop}[{Schmidt \citep{Schmidt:1968}; see also \citep[Chapter 5.3]{Schmidt:1994}}]
\label{prop:SchmidtSolvSglat}\label{prop:SchmidtSglat}For a group
$G$, the following are equivalent:

\begin{enumerate}
\item $G$ is solvable.
\item $L(G)$ has a chain of subgroups $1=G_{0}\subsetneq G_{1}\subsetneq\cdots\subsetneq G_{k}=G$
such that each $G_{i}$ is modular in $L(G)$, and such that each
interval $[G_{i},G_{i+1}]$ is a modular lattice.
\item $L(G)$ has a chain of subgroups $1=G_{0}\subsetdot G_{1}\subsetdot\cdots\subsetdot G_{n}=G$
such that each $G_{i}$ is modular in the interval $[1,G_{i+1}]$.
\end{enumerate}
\end{prop}
The reader will recognize the conditions in Proposition~\ref{prop:SchmidtSolvSglat}
as direct analogues of the conditions from Section~\ref{subsec:PrelimGrps}.
Despite this close correspondence, proving that Conditions (2) and
(3) of the proposition imply solvability is not at all trivial.

Although Proposition~\ref{prop:SchmidtSolvSglat} combinatorially
characterizes solvable subgroup lattices, we find it somewhat unsatisfactory.
We don't know how to use the implicit lattice conditions to calculate
Möbius numbers. And, as the following example will show, lattices
satisfying the implicit conditions need not be shellable.

\begin{figure}
\includegraphics[scale=0.8]{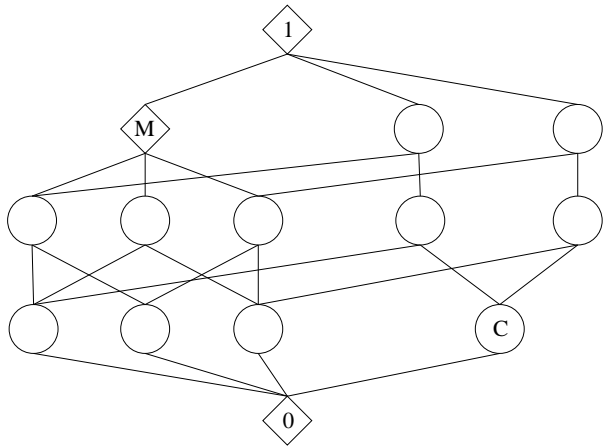}

\caption{\label{fig:NonshellModbymod}A non-shellable lattice satisfying Conditions
(2) and (3) of Proposition~\ref{prop:SchmidtSglat}.}
\end{figure}
 
\begin{example}
Consider the lattice whose Hasse diagram is pictured in Figure~\ref{fig:NonshellModbymod}.
The element $M$ is easily verified to be modular in this lattice,
and the interval $[\hat{0},M]$ is a Boolean lattice, hence modular.
But since the interval $[C,\hat{1}]$ is disconnected, the lattice
is not shellable or Cohen-Macaulay.
\end{example}
It is our opinion that a good definition of ``solvable lattice''
should be equivalent to solvability on subgroup lattices, and obey
as many of the useful properties possessed by supersolvable lattices
as possible. Among these are: $EL$-shellability, efficient computation
of homotopy type and/or homology bases and/or Möbius numbers, and
self-duality of the property. Perhaps even more importantly, such
a definition should have many combinatorial examples.

We believe comodernistic lattices to be an important step towards
understanding a definition or definitions of ``solvable lattice''.
As Theorem~\ref{thm:SgLatComodern} states, comodernism is equivalent
to solvability on subgroup lattices. While we do not know if a comodernistic
lattice is always $EL$-shellable, we have shown such a lattice to
be $CL$-shellable. The $CL$-labeling allows efficient calculation
of homotopy type, and consequences thereof. Perhaps most importantly,
there are many natural examples of comodernistic lattices.

Unfortunately, comodernism is not a self-dual property, as is made
clear by Example~\ref{exa:OrdcongN5}. 

\subsection{Proof of Theorem~\ref{thm:SgLatComodern}}

We begin the proof by reviewing a few elementary facts about groups
and subgroup lattices, all of which can be found in \citep{Schmidt:1994},
or easily verified by the reader. We say that $H$ \emph{permutes}
with $K$ if $HK=KH$. 
\begin{lem}
\label{lem:PermutableProps}Let $H,K,$ and $L$ be subgroups of a
group $G$.

\begin{enumerate}
\item If $H$ permutes with $K$, then $HK=KH$ is the join in $L(G)$ of
$H$ and $K$.
\item If $N\normalin G$, then $N$ permutes with every subgroup of $G$.
\item (Dedekind Identity) If $H\subseteq K$, then $H(K\cap L)=K\cap HL$
and $(K\cap L)H=K\cap LH$.
\end{enumerate}
\end{lem}
\begin{cor}
If $K\supset H$ permutes with every subgroup on the interval $[H,G]$,
then $K$ is a modular element of this interval.
\end{cor}
By Lemma~\ref{lem:LmMaxlIsMod}, the elements of a sub-$M$-chain
of a comodernistic lattice satisfy the modularity condition as in
part (3) of Proposition~\ref{prop:SchmidtSglat}. It follows immediately
by the same proposition that $G$ is solvable if $L(G)$ is comodernistic. 

For the other direction, every subgroup of a solvable group is solvable.
Thus, it suffices to find a modular coatom in the interval $[H,G]$
over any subgroup $H$. Let $1=N_{0}\subsetneq N_{1}\subsetneq\cdots\subsetneq N_{k}=G$
be a chief series. It follows that each $HN_{i}$ is a subgroup for
each $i$. Let $\ell$ be the maximal index such that $HN_{\ell}<G$,
and let $K$ be any coatom of the interval $[HN_{\ell},G]$. 

We will show that $K$ permutes with every subgroup $L$ on $[H,G]$,
hence is modular on the same interval. For any such $L$, we have
\begin{align*}
H(L\cap N_{\ell+1}) & =L\cap HN_{\ell+1}=L\cap G=L,\mbox{ and similarly}\\
(L\cap N_{\ell+1})H & =L\cap N_{\ell+1}H=L\cap G=L,
\end{align*}
so that $H$ permutes with $L\cap N_{\ell+1}$. Moreover, $N_{\ell}\subseteq K\cap N_{\ell+1}\subseteq N_{\ell+1}$,
and since $N_{\ell+1}/N_{\ell}$ is abelian, the Correspondence Theorem
gives that $K\cap N_{\ell+1}\normalin N_{\ell+1}$. Thus, $K\cap N_{\ell+1}$
permutes with $L\cap N_{\ell+1}$. Now 
\[
KL=H(K\cap N_{\ell+1})H(L\cap N_{\ell+1})=H(L\cap N_{\ell+1})H(K\cap N_{\ell+1})=LK,
\]
 as desired.

\subsection{Homotopy type of the subgroup lattice of a solvable group}

It is immediate from Lemma~\ref{lem:PermutableProps} that if $N\normalin G$
then $HN$ permutes with all subgroups on the interval $[H,G]$. Thus,
a chief series $\{N_{i}\}$ lifts to a chain of left-modular elements
$\{HN_{i}\}$ on the interval $[H,G]$. In a solvable group we can
(by the proof of Theorem~\ref{thm:SgLatComodern}) complete the chain
$\{HN_{i}\}$ to a sub-$M$-chain. Let $\lambda$ be constructed according
to this choice of sub-$M$-chain in all applicable intervals, and
consider the decreasing chains of $\lambda$.

It is immediate by basic facts about left-modular elements that a
decreasing maximal chain on $L(G)$ contains as a subset a chain of
complements to the chief series $\{N_{i}\}$. Since a chain of complements
to $\{N_{i}\}$ in a solvable group is a maximal chain \citep[Lemma 9.10]{Doerk/Hawkes:1992},
such chains are exactly the decreasing maximal chains. 

The order complex of a $CL$-shellable poset is a bouquet of spheres,
where the spheres are in bijective correspondence with the decreasing
chains of the poset. Thus, we recover the homotopy-type calculation
of \citep{Thevenaz:1985} (see also \citep{Woodroofe:2008,Woodroofe:2012a}). 

\section{\label{sec:kEqualVariations}$k$-equal partition and related lattices}

In this section, we will show that the $k$-equal partition lattices
are comodernistic. We'll also show two related families of lattices
to be comodernistic.

\subsection{\label{subsec:k-equal}$k$-equal partition lattices}

Recall that the \emph{$k$-equal partition lattice} $\Pi_{n,k}$ is
the subposet of the partition lattice $\Pi_{n}$ consisting of all
partitions whose non-singleton blocks have size at least $k$. 
\begin{thm}
\label{thm:kEqComod}For any $1\leq k\leq n$, the $k$-equal partition
lattice $\Pi_{n,k}$ is comodernistic.
\end{thm}
\begin{proof}
Consider an interval $[\pi',\pi]$ in $\Pi_{n,k}$. Let $\pi$ be
$C_{1}\,\vert\,C_{2}\,\vert\,\dots\,\vert\,C_{m}$, and assume without
loss of generality that $C_{1}$ is not a block in $\pi'$. Then $C_{1}$
is formed by merging blocks $B_{1},\dots,B_{\ell}$ of $\pi'$. Suppose
that the $B_{i}$'s are ordered by increasing size, so that $\left|B_{1}\right|\leq\left|B_{2}\right|\leq\cdots$.
Consider the element 
\[
m=B_{1}\,\vert\,B_{2}\cup\dots\cup B_{\ell}\,\vert\,C_{2}\,\dots\,\vert\,C_{m}.
\]
Suppose that $\sigma$ is some partition on $[\pi',\pi]$, and $D$
is the block of $\sigma$ containing $B_{1}$. If $D=B_{1}$, then
$\sigma\wedge m=\sigma$. Otherwise, there are two cases:

\emph{Case }1: $\left|B_{1}\right|>1$. Then $\sigma\wedge m$ is
formed by splitting $D$ into blocks $B_{1}$, $D\setminus B_{1}$.
Notice that $\left|D\setminus B_{1}\right|\geq\left|B_{1}\right|\geq k$
by the ordering of the $B_{i}$'s. 

\emph{Case }2: $\left|B_{1}\right|=1$. Then if $\left|D\right|>k$,
then $\sigma\wedge m$ is formed by splitting $D$ into smaller blocks
$B_{1},D\setminus B_{1}$. Otherwise, we have $\left|D\right|=k$,
and $\sigma\wedge m$ is formed by splitting $D$ into $k$ singletons.

In either situation, we have $\sigma\wedge m\lessdot\sigma$, so Lemma~\ref{lem:LMcoatomCond}
gives $m$ to be left-modular on the desired interval.
\end{proof}
We recover from Theorem~\ref{thm:kEqComod} a weaker form of the
result \citep[Theorem 6.1]{Bjorner/Wachs:1996} that $k$-equal partition
lattices are $EL$-shellable. Repeated application of Lemma~\ref{lem:ComputingMobius}
recovers the same set of decreasing chains for the comodernistic labeling
as in \citep[Corollary 6.2]{Bjorner/Wachs:1996}. See also \citep{Bjorner/Welker:1995}.

\subsection{\label{subsec:khequalTypeB}$k,h$-equal partition lattices in type
$B$}

By a \emph{sign pattern} of a set $S$, we refer to an assignment
of $+$ or $-$ to each element of $S$, considered up to reversing
the signs of every element of $S$. Thus, if $S$ has an order, an
equivalent notion is to assign a $+$ to the first element of $S$
and either $+$ or $-$ to each remaining element.

A \emph{signed partition} of $\{0,1,\dots,n\}$ then consists of a
partition of the set, together with a sign pattern assignment for
each block not containing $0$. The block containing $0$ is called
the \emph{zero block}, and other blocks are called \emph{signed blocks}.
The \emph{signed partition lattice $\Pi_{n}^{B}$} consists of all
signed partitions of $\{0,1,\dots,n\}$. The cover relations in $\Pi_{n}^{B}$
are of two types: merging two signed blocks, and selecting one of
the two possible patterns of the merged set; or merging a signed block
with the zero block (thereby `forgetting' the sign pattern on the
signed block).

The signed partition lattice is well-known to be supersolvable. Indeed,
if $\pi$ is a signed partition where every signed block is a singleton,
then $\pi$ is left-modular in $\Pi_{n}^{B}$. 

Björner and Sagan \citep{Bjorner/Sagan:1996} considered the \emph{signed
$k,h$-equal partition lattice}, where $1\leq h<k\leq n$. This is
the subposet $\Pi_{n,k,h}^{B}$ consisting of all signed partitions
whose non-singleton signed blocks have size at least $k$, and whose
zero block is either a singleton or has size at least $h+1$. 
\begin{thm}
\label{thm:khEqComod}For any $1\leq h<k\leq n$, the signed $k,h$-equal
partition lattice $\Pi_{n,k,h}^{B}$ is comodernistic.
\end{thm}
\begin{proof}
We proceed similarly to the proof of Theorem~\ref{thm:kEqComod}.
Let $[\pi',\pi]$ be an interval in $\Pi_{n,k,h}^{B}$, and $\pi$
be $C_{1}\,\vert\,C_{2}\,\vert\,\dots\,\vert\,C_{m}$. Assume without
loss of generality that $C_{1}$ is not a block in $\pi'$. Then $C_{1}$
is formed by merging blocks $B_{1},\dots,B_{\ell}$ of $\pi'$. Let
$B_{1}$ be the smallest signed block in this list, and consider the
element 
\[
m=B_{1}\,\vert\,B_{2}\cup\dots\cup B_{\ell}\,\vert\,C_{2}\,\dots\,\vert\,C_{m}.
\]
 Now let $\sigma$ be some partition in the interval $[\pi',\pi]$,
and $D$ be the block of $\sigma$ containing $B_{1}$. If $D=B_{1}$,
then $\sigma\wedge m=\sigma$. Otherwise, there are two cases:

\emph{Case }1: $\left|B_{1}\right|>1$. Since $\sigma$ is on $[\pi',\pi]$,
we see that $D\subseteq C_{1}$. If $D$ is a signed block, then $\left|D\setminus B_{1}\right|\geq\left|B_{1}\right|\geq k$
by the ordering of the blocks. Otherwise, by choice of $B_{1}$, no
part of $\sigma\wedge m$ is a signed singleton block contained in
$C_{1}$. It follows that $\left|D\setminus B_{1}\right|\geq h+1$.
In either situation, it follows that $\sigma\wedge m$ is formed by
splitting $D$ into blocks $B_{1}$, $D\setminus B_{1}$.

\emph{Case }2: $\left|B_{1}\right|=1$. If $\left|D\right|>k$, then
$\sigma\wedge m$ is formed by splitting $D$ into smaller blocks
$B_{1},D\setminus B_{1}$. Similarly if $0\in D$ and $\left|D\right|>h+1$.
Otherwise, we have $\left|D\right|=k$ or $\left|D\right|=h$ (depending
on whether $0\in D$), and $\sigma\wedge m$ is formed by splitting
$D$ into singletons.

In either case, we have $\sigma\wedge m\lessdot\sigma$, hence that
$m$ is left-modular on the desired interval by Lemma~\ref{lem:LMcoatomCond}.
\end{proof}
Björner and Sagan \citep[Theorem 4.4]{Bjorner/Sagan:1996} showed
$\Pi_{n,k,h}^{B}$ to be $EL$-shellable. We recover from Theorem~\ref{thm:khEqComod}
the weaker result of $CL$-shellability. However, we remark that our
proof is significantly simpler, and still allows easy computation
of a cohomology basis, etc.

There is also a ``type $D$ analogue'' of $\Pi_{n,k}$ and $\Pi_{n,k,h}^{B}$.
Björner and Sagan considered this lattice in \citep{Bjorner/Sagan:1996},
but left the question of shellability open. Feichtner and Kozlov gave
a partial answer to the type $D$ shellability question in \citep{Feichtner/Kozlov:2000}. 

Our basic technique in the proofs of Theorems~\ref{thm:kEqComod}
and \ref{thm:khEqComod} is to show that left-modularity of coatoms
in $\Pi_{n}$ and $\Pi_{n}^{B}$ is sometimes inherited in the join
subsemilattices $\Pi_{n,k}$ and $\Pi_{n,k,h}^{B}$. It is easy to
verify from the (here omitted) definition that the type $D$ analogue
of $\Pi_{n}$ and $\Pi_{n}^{B}$ has no left-modular coatoms, see
also \citep{Hoge/Rohrle:2014}. For this reason, the straightforward
translation of our techniques will not work in type $D$. We leave
open the question of under what circumstances the type $D$ analogue
of $\Pi_{n,k}$ has left-modular coatoms, or is comodernistic. 

\subsection{\label{subsec:AffinityEqPartition}Partition lattices with restricted
element-block size incidences}

\global\long\def\aff{\operatorname{aff}}

The $k$-equal partition lattices admit generalizations in several
directions. One such generalization, examined in \citep{Bjorner/Wachs:1996},
is that of the subposet of partitions where the size of every block
is in some set $T$. Further generalizations in similar directions
are studied in \citep{Ehrenborg/Jung:2013,Ehrenborg/Readdy:2007}.

We consider here a different direction. Motivated by the signed $k,h$-equal
partition lattices, Gottlieb \citep{Gottlieb:2014} examined a related
sublattice of the (unsigned) partition lattice. In Gottlieb's lattice,
the size of a block is restricted to be at least $k$ or at least
$h$, depending on whether or not the block contains a distinguished
element.

We further generalize to allow each element to have a different block-size
restriction associated to it. More formally, we consider a map $\aff:[n]\rightarrow[n]$,
which we consider as providing an \emph{affinity} to each element
$x$ of $[n]$. We consider two subposets of the partition lattice
$\Pi_{n}$:
\begin{align*}
\Pi_{\aff}^{\forall} & \triangleq\left\{ \pi\in\Pi_{n}\,:\,\mbox{every nonsingleton block }B\mbox{ has }\left|B\right|\geq\aff(x)\mbox{ for every }x\in B\right\} ,\mbox{ and }\\
\Pi_{\aff}^{\exists} & \triangleq\left\{ \pi\in\Pi_{n}\,:\,\mbox{every nonsingleton block }B\mbox{ contains some }x\mbox{ such that }\left|B\right|\geq\aff(x)\right\} .
\end{align*}

It is clear that both subposets are join subsemilattices of $\Pi_{n}$,
hence lattices. 
\begin{thm}
\label{thm:affEqComod}For any selection of affinity map $\aff$,
the lattices $\Pi_{\aff}^{\exists}$ and $\Pi_{\aff}^{\forall}$ are
comodernistic.
\end{thm}
\begin{proof}
As in the proof of Theorem~\ref{thm:kEqComod}, consider an interval
$[\pi',\pi]$. Let some block of $\pi$ split nontrivially into blocks
$B_{1},B_{2},\dots,B_{\ell}$ of $\pi'$. As in Theorem~\ref{thm:kEqComod},
assume that the blocks are sorted by increasing size, and in particular
that $\left|B_{1}\right|\le\left|B_{i}\right|$ for all $i$.

Now, if there are multiple singleton blocks 
\[
B_{1}=\{x_{1}\},B_{2}=\{x_{2}\},\dots B_{j}=\{x_{j}\},
\]
then sort these by affinity. In the case of $\Pi_{\aff}^{\exists}$,
let $\aff(x_{1})\geq\aff(x_{2})\geq\dots$; while for $\Pi_{\aff}^{\forall}$
reverse to require $\aff(x_{1})\leq\aff(x_{2})\leq\dots$.

The remainder of the proof now goes through entirely similarly to
that of Theorems~\ref{thm:kEqComod} and \ref{thm:khEqComod}.
\end{proof}
Theorem~\ref{thm:affEqComod} has applications to lower bounds of
the complexity of a certain computational problem, directly analogous
to the work in \citep{Bjorner/Lovasz:1994,Bjorner/Lovasz/Yao:1992}
with $\Pi_{n,k}$.

\bibliographystyle{hamsplain}
\bibliography{7_Users_russw_Documents_Research_mypapers_A_broad_class_of_shellable_lattices_Master}

\end{document}